\documentclass[11pt]{amsart}

\usepackage{amsmath}

\usepackage{amssymb}

\usepackage{amsthm}

\usepackage{color}
\newtheorem{thm}{Theorem}[section]

\newtheorem{cor}{Corollary}[section]

\newtheorem{rem}{Remark}[section]

\newcommand{\bR}{\mathbb{R}}

\newcommand{\bE}{\mathbb{E}}

\newcommand{\bP}{\mathbb{P}}

\newcommand{\al}{\alpha}

\newcommand{\la}{\lambda}

\newcommand{\ep}{\varepsilon}

\newcommand{\goto}{\rightarrow}

\def\<{\left<}\def\>{\right>}
\def\({\left(}\def\){\right)}

\newcommand{\dd}{\mathrm{d}}

\begin{document}

%\date

\centerline{\Large On the last exit times for spectrally negative L\'evy processes }

\vskip 1 cm

\centerline{\large Yingqiu Li$^{a}$,  Chuancun Yin$^{b}$ and Xiaowen Zhou$^{c}$ }

\bigskip

\centerline{\small $^{a}$School of Mathematics and Statistics, Changsha University of Science and Technology}

\smallskip

\centerline{\small $^{b}$School of Statistics,   Qufu Normal University}

 \smallskip

\centerline{\small $^{c}$Department of Mathematics and Statistics, Concordia University}
%\centerline{\small 1455 de Maisonneuve Blvd. West, Montreal, Quebec, H3G 1M8, Canada}

\bigskip

\vskip 1 cm

\centerline{\large Abstract}

\vskip 0.5 cm

Using a new approach, for spectrally negative L\'evy processes we find joint Laplace transforms involving the last exit time (from a semi-infinite interval), the value of the process at the last exit time and the associated occupation time, which generalizes some previous results.

\vskip 0.5 cm



\section{Introduction}

The first passage times, also known as  the first exit times, have been studied extensively for stochastic processes.
%For spectrally negative L\'evy processes the first passage time related quantities can often be expressed in terms of scale functions for the associated spectrally negative L\'evy processes.
 They find many applications in risk theory since the first exit time corresponds to the ruin time of a risk process. The joint distribution of quantities related to the ruin time, often referred as Gerber-Shiu functions, are very interested in the study of risk models. For spectrally negative L\'evy processes those first passage time related quantities can often be expressed in terms of the associated scale functions.

The last passage time is also interested in risk theory. In a risk model, it can be regarded as the final recovery time after when there will be no more ruin; see \cite{Ger09} for early work in this respect.
A  Laplace transform for the last exit time from the positive part of the real line  is first obtained in \cite{CY05} for spectrally negative L\'evy process using an identity between density of the first exit time and marginal density of the L\'evy process.
Applying the strong Markov property and the memoryless property for exponential random variables, in \cite{Bau09} Laplace transforms are also found for the last exit times (before an independent exponential time) from the positive part and the negative part of the real line, respectively. In both \cite{CY05} and \cite{Bau09} the Laplace transforms are expressed in terms of the Laplace exponent and scale functions for the associated spectrally negative L\'evy process.

In this paper we further consider the joint distribution concerning the last time when a spectrally negative L\'evy process leaves a semi-infinite interval. More precisely, for a spectrally negative L\'evy process running up to an exponential time, we find the joint Laplace transform of the occupation time until the last exit time, the last exit time and value of the process at the last exit time.    They can be treated as Gerber-Shiu function like quantities for the last exit time of a spectrally negative L\'evy process and  generalize the previous results in \cite{CY05} and \cite{Bau09}. Right before finalizing this paper we notice the recent work of \cite{LC16} in which the joint distribution involving the the last exit time, the occupation time up to the last exit time and the value of the process prior to the last exit time. \cite{LC16} adopts a time reversal argument and uses a previous result on occupation time.

Our approach is different from those  in \cite{CY05}, \cite{Bau09} and \cite{LC16}.
We use the Poisson approach of \cite{LZ14} and \cite{AIZ16} to find  Laplace transforms of  the occupation times.
To overcome difficulties due to the non-stoping last exit time  and the  infinite activity of the L\'evy processes, we adopt
 a perturbation approach, similar to that of \cite{Zhou07}, to approximate the last exit times by first exit times over certain events. We then apply the fluctuation theory for spectrally negative L\'evy processes to express the approximating Laplace transforms in terms of scale functions and Laplace exponents before taking limits to reach the desired results. We also remark that whenever there is a jump at the last exit time, our result can not be directly obtained from that in  \cite{LC16} since the last exit time is not a stopping time although the approach of \cite{LC16} can still be used to derive our results. We illustrate this with an example at the end of the paper. We also discuss the possible creeping behavior at the last exit time that is not covered in \cite{LC16}.

In  the next section we present some preliminary results for spectrally negative L\'evy processes. The main results and proofs are provided in Section  \ref{main}. Some corollaries, discussions together with examples are also given in this section.

\section{Preliminaries}\label{pre}
In this section we collect some known fluctuation results for  spectrally negative L\'evy processes to prepare for the future proofs.
Let process $X$ be a spectrally negative L\'evy process, i.e. a L\'evy process with right continuous sample paths and  with no positive jumps. Then for any
$t>0$,
\[\frac{1}{t}\log\bE e^{\lambda X_t}=\psi(\lambda) \]
 with
\[\psi(\lambda)=a\lambda+\frac{1}{2}\sigma^2\la^2+\int_{-\infty}^0 (e^{\la x}-1-\la x1_{x>-1})\Pi(dx),\]
where $\Pi$ is a $\sigma$-finite measure on $(-\infty, 0)$ satisfying $\int_{-\infty}^0 (x^2\wedge 1) \Pi(dx)<\infty$. Write $\Phi: [0,\infty)\mapsto [0,\infty)$ for the inverse function of $\psi$.

Scale functions play a key role in analyzing spectrally negative
L\'evy processes. The scale function $W^{(p)}$ for $X$ is a nonnegative, increasing function specified
by the following Laplace transform. For any $q\geq 0$,
\[\int_0^\infty e^{-\lambda x}W^{(q)}(x)dx=\frac{1}{\psi(\lambda)-q}, \,\,\, \lambda>\Phi(q),\]
and $W^{(q)}(x):=0$ for $x<0$.
We also need the scale functions
\[Z^{(p)}(x):=1+p\int_0^x W^{(p)}(y)dy\]
 and
\[Z^{(p)}(x,\theta):=e^{\theta x}\left(1+(p-\psi(\theta))\int_0^x e^{-\theta y}W^{(p)}(y)dy \right).\]
with $Z^{(p)}(x):=1$ and $Z^{(p)}(x,\theta):=e^{\theta x} $ for $x<0$, respectively. Write $W(x), Z(x), Z(x,\theta)$ for $W^{(0)}(x), Z^{(0)}(x), Z^{(0)}(x,\theta)$, respectively.

For all
$p, q\geq 0$, if $X$ is of bounded variation, then $W^{(p)}(0)=W^{(q)}(0)$; if $X$ is of unbounded variation, then
$$\lim_{\ep\goto 0+}\frac{W^{(p)}(\ep)}{W^{(q)}(\ep)}=1; $$
 see \cite{Ky14}.

For any $a\in\bR$, define one-sided exit times
\[\tau^+_a:=\inf\{t>0: X_t>a\}\]
and
\[\tau^-_a:=\inf\{t>0: X_t<a\}\]
with the convention $\inf\emptyset=\infty$.

We collect some fluctuation  identities for $X$ concerning the solutions to the exit problems and potential measures, which can be found in \cite{KKR13} and \cite{Ky14}. For event $A$ and random variable $Y$, write $\bE [Y; A]=\bE Y1_A$.

We first present solutions to the exit problems.
\[\bE [e^{-q\tau^+_a};  \tau^+_a<\infty]=e^{-\Phi(q)a}, \]
\[ \bE_x [e^{-q\tau^-_0};  \tau^-_0<\infty]=Z^{(q)}(x)-\frac{q}{\Phi(q)}W^{(q)}(x),\]
\begin{equation}\label{idenA}
\begin{split}
\bE_x [e^{-u\tau^-_0+vX_{\tau^-_0}}; {\tau^-_0<\infty}]&=Z^{(u)}(x,v)-\frac{u-\psi(v)}{\Phi(u)-v}W^{(u)}(x) \\
&=e^{vx}\left(1+(u-\psi(v))\int_0^x e^{-vy}W^{(u)}(y)\dd y-\frac{u-\psi(v)}{\Phi(u)-v}e^{-vx}W^{(u)}(x) \right)
\end{split}
\end{equation}
and
\[\bE_x [e^{-u\tau^-_0+\Phi(u)X_{\tau^-_0}}; {\tau^-_0<\infty}]=e^{\Phi(u)x}-\frac{1}{\Phi'(u)}W^{(u)}(x), \]
\begin{equation}\label{idenD}
\begin{split}
&\bE_x [ e^{-u\tau^-_0+vX_{\tau^-_0}}; {\tau^-_0<\tau^+_a}]\\
&=Z^{(u)}(x,v)-\frac{u-\psi(v)}{\Phi(u)-v}W^{(u)}(x)-\frac{W^{(u)}(x)}{W^{(u)}(a)}\left(Z^{(u)}(a,v)-\frac{u-\psi(v)}{\Phi(u)-v}W^{(u)}(a) \right) \\
&=Z^{(u)}(x,v)-\frac{W^{(u)}(x)}{W^{(u)}(a)}Z^{(u)}(a,v).
\end{split}
\end{equation}

We then present expressions for potential measures. For $a>0$,
\begin{equation}\label{idenB}
\begin{split}
\int_0^\infty e^{-qt} \bP\{X_t\in dy, t<\tau^+_a\}dt=\left( e^{-\Phi(q)a}W^{(q)}(a-y)-W^{(q)}(-y)\right)dy.
\end{split}
\end{equation}
For $a\leq 0$,
\begin{equation}\label{idenC}
\begin{split}
\int_0^\infty e^{-qt} \bP\{X_t\in dy, t<\tau^-_a\}dt=\left(e^{-\Phi(q)(y-a)}W^{(q)}(-a)-W^{(q)}(-y) \right)dy.
\end{split}
\end{equation}

%For the next result we need to introduce a new notion.

%For $p, q\geq 0$ and $0\leq a\leq x$, let
%\begin{equation}\label{thm3_b}
%\begin{split}
%W_a^{(p,q)}(x)
%&:=W^{(p)}(x)+(q-p)\int_a^x W^{(q)}(x-y)W^{(p)}(y)dy\\
%&=W^{(q)}(x)+(p-q)\int_0^a W^{(q)}(x-y)W^{(p)}(y)dy.\\
%\end{split}
%\end{equation}
%Then for $a\leq x\leq b$, we have
%\[\bE_x [e^{-p\tau^-_a}W^{(q)}(X_{\tau^-_a});\tau^-_a<\tau^+_b ]=W_a^{(q,p)}(x)-\frac{W^{(p)}(x-a)}{W^{(p)}(b-a)}W_a^{(q,p)}(b);  \]
%see \cite{LRZ14}.
%Letting $b\goto\infty$, then
%\begin{equation*}
%\begin{split}
%&\bE_x [e^{-p\tau^-_a}W^{(q)}(X_{\tau^-_a});\tau^-_a<\infty ]\\
%&=W_a^{(q,p)}(x)-W^{(p)}(x-a)\left(e^{\Phi(p)a}+(q-p) \int_0^a e^{\Phi(p)(a-y)} W^{(q)}(y)dy\right),
%\end{split}
%\end{equation*}
%where we have applied identity (\ref{thm3_b}). In particular, for $p=q$,
%\[\bE_x [e^{-p\tau^-_a}W^{(p)}(X_{\tau^-_a});\tau^-_a<\infty ]=W^{(p)}(x)-e^{\Phi(p)a}W^{(p)}(x-a). \]

For an independent exponential random variable $e_r$ with rate $r$, let
\[T^+(r):=\sup\{0<t\leq  e_r: X_t< 0 \}\]
and
\[T^-(r):=\sup\{0<t\leq  e_r : X_t> 0 \}\]
 with the convention $\sup\emptyset=0$. Then $T^+(r)$ and $T^-(r)$ are, respectively,  the last times of exiting the negative and positive part of the
 real line up to the time $e_r$. Write $T^+$ and $T^-$ for $T^+(0)$ and $T^-(0)$, respectively.

In this paper,  for constants $p, q, \theta, r >0$ and $x\in\bR$,  we want to find expressions for  $$\bE_x[e^{-pT^+(r)-q \int_0^{T^+(r)} 1_{X_s <0} ds+\theta X(T^+(r))}; T^+(r)=e_r],$$
 $$\bE_x[e^{-pT^+(r)-q \int_0^{T^+(r)} 1_{X_s <0} ds}; T^+(r)<e_r]$$
 and
 \[\bE_x [e^{-pT^-(r)+\theta X(T^-(r))-q \int_0^{T^-(r)}  1_{X_s >0} ds}; T^-(r)<e_r],\]
 \[\bE_x [e^{-pT^-(r)-q \int_0^{T^-(r)}  1_{X_s >0} ds-\theta X(T^-(r))}; T^-(r)=e_r]. \]

\section{Main Results}\label{main}

We first proceed to find expressions for
\[\omega^+_1(x):=\bE_x[e^{-pT^+(r)-q \int_0^{T^+(r)} 1_{X_s \leq 0} ds}; T^+(r)<e_r]\]
and
\[\omega^+_2(x):=\bE_x[e^{-pT^+(r)-q \int_0^{T^+(r)} 1_{X_s \leq 0} ds+\theta X(T^+(r))}; T^+(r)=e_r]. \]

\begin{thm}
For any $p, q, r, \theta>0$ and $x\in\bR$, we have
 \begin{equation}\label{thm1a}
\begin{split}
\omega^+_1(x)&=\frac{r(W^{(r)}(x)-W^{(p+r)}(x))}{\Phi(r)}+\frac{r(\Phi(p+q+r)-\Phi(p+r))}{q\Phi(r)}Z^{(p+r)}(x,\Phi(p+q+r))\\
&\quad-r\int_0^x W^{(r)}(y)dy
 \end{split}
\end{equation}
and
\begin{equation}\label{thm1b}
\begin{split}
\omega^+_2(x)
&=\frac{r}{p+q+r-\psi(\theta)}\left(-\frac{p+r-\psi(\theta)}{q}\times\frac{\Phi(p+q+r)-\Phi(p+r)}{\Phi(p+r)}\right) Z^{(p+r)}(x, \Phi(p+q+r))\\
&\quad +\frac{r}{p+q+r-\psi(\theta)}Z^{(p+r)}(x,\theta).\\
%+\frac{r(p+r-\psi(\theta))}{p+q+r-\psi(\theta)}\left( \frac{1}{\Phi(p+r)}-\frac{1}{\Phi(p+r)-\theta}\right)W^{(p+r)}(x).  \\
 \end{split}
\end{equation}
\end{thm}

\begin{proof}
The proof is carried out in two steps.
 We first assume that $X$ is of bounded variation.  Write $e_q$ for an independent  exponential random variable.
Let $T_1<T_2<\ldots$ be the sequence of arrival times for an
independent Poisson process with rate $q$. Let
\[A^+:=\cap_{i: T_i<T^+(r) }\{ X_{T_i}> 0 \}.\]
Then
\[\omega^+_1(x)= \bE_x\left[ e^{-pT^+(r)}; T^+(r)<e_r,  A^+\right]\]
and
\[\omega^+_2(x)= \bE_x\left[ e^{-pT^+(r)+\theta X(e_r)}; T^+(r)=e_r,  A^+\right]\]
by a property of Poisson process; see a similar argument in \cite{LZ14}.
Then conditioning on $T_1$, by the the memoryless property for exponential distribution and the strong Markov property,
 \begin{equation}\label{equA1}
\begin{split}
\omega^+_1(0)&=\bP\{ \tau^-_0>e_r\}+\int_{-\infty}^0 \bE[e^{-p\tau^-_0}; \tau^-_0< e_r, X_{\tau^-_0}\in dy]
\bE_y[e^{-p\tau^+_0}; \tau^+_0<e_q\wedge e_r]\omega^+_1(0)  \\
&=1-\bE e^{-r\tau^-_0} +\int_{-\infty}^0 \bE[e^{-(p+r)\tau^-_0};  X_{\tau^-_0}\in dy]\bE_y e^{-(p+q+r)\tau^+_0}\omega^+_1(0) \\
&= \frac{r}{\Phi(r)}W^{(r)}(0)+\bE e^{-(p+r)\tau^-_0+\Phi(p+q+r) X(\tau^-_0)}\omega^+_1(0) \\
&=\frac{r}{\Phi(r)}W^{(r)}(0)+\left(1-\frac{q}{\Phi(p+q+r)-\Phi(p+r)}W^{(p+r)}(0) \right)\omega^+_1(0),
 \end{split}
\end{equation}
where we have used the identity (\ref{idenA}). Solving (\ref{equA1}) for $\omega^+_1(0)$ we have
\begin{equation}\label{equA3}
\begin{split}
\omega^+_1(0)&= \frac{r}{q\Phi(r)}(\Phi(p+q+r)-\Phi(p+r)).\\
\end{split}
\end{equation}
%where we need the fact that $ W^{(r)}(0)=W^{(p+r)}(0)$; see \cite{Ky14}.

Similarly,
 \begin{equation}\label{equA2}
\begin{split}
\omega^+_2(0)&=\int_{-\infty}^0 \bE[e^{-p\tau^-_0}; \tau^-_0<e_r, X_{\tau^-_0}\in dy]\\
&\qquad\quad\times\left(\bE_y[e^{-p\tau^+_0}; \tau^+_0<e_q\wedge e_r]\omega^+_2(0)+\bE_y[e^{-pe_r+\theta X(e_r)}; e_r<\tau^+_0\wedge e_q]\right) \\
&=\bE e^{-(p+r)\tau^-_0+\Phi(p+q+r) X(\tau^-_0)}\omega^+_2(0)\\
&\quad+\frac{r}{p+q+r-\psi(\theta)}\left(\bE e^{-(p+r)\tau^-_0+\theta X(\tau^-_0)}-\bE e^{-(p+r)\tau^-_0+\Phi(p+q+r) X(\tau^-_0)} \right) \\
&=\left(1-\frac{q}{\Phi(p+q+r)-\Phi(p+r)}W^{(p+r)}(0) \right)\omega^+_2(0)\\
%&\quad+\frac{r}{p+q+r-\psi(\theta)}\left(\frac{q}{\Phi(p+q+r)-\Phi(p+r)}W^{(p+r)}(0) -\frac{p+r}{\Phi(p+r)}W^{(p+r)}(0)  \right),
&\quad+\frac{r}{p+q+r-\psi(\theta)}\left(\frac{q}{\Phi(p+q+r)-\Phi(p+r)}W^{(p+r)}(0)-\frac{p+r-\psi(\theta)}{\Phi(p+r)-\theta}W^{(p+r)}(0) \right),\\
\end{split}
\end{equation}
where for $y<0$ and $\theta>p+q+r$, by identity (\ref{idenB}) we have
\begin{equation*}
\begin{split}
\bE_y[e^{-pe_r+\theta X(e_r)}; e_r<\tau^+_0\wedge e_q]
&=\bE_y[e^{-(p+q)e_r+\theta X(e_r)}; e_r<\tau^+_0]\\
&=\int_{-\infty}^0 e^{\theta z}r\left( e^{\Phi(p+q+r)y}W^{(p+q+r)}(-z)-W^{(p+q+r)}(y-z)\right)\\
&=e^{\Phi(p+q+r)y}\frac{r}{\psi(\theta)-p-q-r}-e^{\theta y}\frac{r}{\psi(\theta)-p-q-r}.\\
\end{split}
\end{equation*}
Solving (\ref{equA2}) for $\omega^+_2(0)$ we have
\begin{equation}
\begin{split}
\omega^+_2(0)
%&=\frac{\frac{r}{p+q+r-\psi(\theta)}\left(\frac{q}{\Phi(p+q+r)-\Phi(p+r)}
%-\frac{p+r-\psi(\theta)}{\Phi(p+r)-\theta}  \right)}{\frac{q}{\Phi(p+q+r)-\Phi(p+r)}}\\
&=\frac{r}{p+q+r-\psi(\theta)}\left(1-\frac{p+r-\psi(\theta)}{q}\times\frac{\Phi(p+q+r)-\Phi(p+r)}{\Phi(p+r)-\theta}\right). \\
\end{split}
\end{equation}

We now consider process $X$ with  sample paths of unbounded variation.  For any $\epsilon>0$, let
\[T^+_\ep(r):=\inf\{T^+(r)\leq s\leq e_r: X_s> \ep   \} \]
with the convention $\inf\emptyset= e_r$. Write $T^+_\ep=T^+_\ep(0) $. Then either $T^+(r)<e_r $ or $T^+(r)=e_r $, we have $T^+_\ep(r)\downarrow T^+(r)$ as $\ep\goto 0+ $. Write
%\[\omega^+_1(x):= \bE_x [e^{-pT^+(r)-q \int_0^{T^+(r)} 1_{X_s \leq 0} ds};T^+(r)<e_r ],  \]
%\[\omega^+_2(x):= \bE_x [e^{-pT^+(r)-q \int_0^{T^+(r)} 1_{X_s \leq 0} ds+\theta X(e_r)};T^+(r)=e_r ],  \]
\[\omega^+_{1,\ep}(x):= \bE_x [e^{-pT^+_\ep(r)-q \int_0^{T^+_\ep(r)} 1_{X_s \leq 0} ds};T^+_\ep(r)<e_r]\]
and
\[\omega^+_{2,\ep}(x):= \bE_x [e^{-pT^+_\ep(r)-q \int_0^{T^+_\ep(r)} 1_{X_s \leq 0} ds+\theta X(e_r)};T^+_\ep(r)=e_r].\]
Then for
\[A^+_\ep:=\cap_{i: T_i<T^+_\ep(r) }\{ X_{T_i}> 0 \},\]
\[\omega^+_{1,\ep}(x)= \bE_x\left[ e^{-pT^+_\ep(r)}; T^+_\ep(r)<e_r, A^+_\ep\right],\]
\[\omega^+_{2,\ep}(x)= \bE_x\left[ e^{-pT^+_\ep(r)+\theta X(e_r)}; T^+_\ep(r)=e_r, A^+_\ep\right],\]
 by the property of Poisson process. It follows that
 \[ \lim_{\ep\goto 0+}\omega^+_{i,\ep}(x)=\omega^+_i(x), \,\,\,\, i=1,2.  \]

 Notice that by (\ref{idenB})
\begin{equation}
\begin{split}
&\bE[e^{-pT^+_\ep(r)}; T^+_\ep(r)<e_r, T_1<\tau^+_\ep\wedge e_r, A^+]\\
&\leq \bP\{  T_1<\tau^+_\ep, 0\leq X_{T_1}\leq \ep\}\\
&=o(W(\ep)).
\end{split}
\end{equation}
 Comparing the first Poisson arrival time $T_1$ with $\tau^+_\ep $ and $e_r $ we have
\begin{equation}\label{equB1}
\begin{split}
\omega^+_{1,\ep}(0)&=\bE[e^{-p\tau^+_\ep};\tau^+_\ep<e_q\wedge e_r]\left(\bP_\ep[\tau^-_0>e_r\}\right.\\
&\qquad\qquad\left.+\int_{-\infty}^0 \bE_\ep[e^{-p\tau^-_0}; \tau^-_0<e_r, X_{\tau^-_0}\in dy]\bE_y[e^{-p\tau^+_0}; \tau^+_0<e_q\wedge e_r]\omega^+_{1,\ep}(0)\right)  \\
&\quad +o(W(\ep))\\
&=e^{-\Phi(p+q+r)\ep}\left\{1-Z^{(r)}(\ep)+\frac{r}{\Phi(r)}W^{(r)}(\ep)+\bE_\ep e^{-(p+r)\tau^-_0+\Phi(p+q+r) X(\tau^-_0)}\omega^+_{1,\ep}(0) \right\}\\
&\qquad+o(W(\ep)) \\
%&=e^{-\Phi(p+q+r)\ep}\frac{r}{\Phi(r)}W^{(r)}(\ep) +o(W(\ep)) \\
%&\quad+\left(1-q\int_0^\ep e^{-\Phi(p+q+r)y}W^{(p+r)}(y)dy \right. \\
%&\qquad\qquad\left.-\frac{q}{\Phi(p+q+r)-\Phi(p+r)}e^{-\Phi(p+q+r)\ep}W^{(p+r)}(\ep) \right)\omega^+_{1,\ep}(0) \\
&=e^{-\Phi(p+q+r)\ep}\frac{r}{\Phi(r)}W^{(r)}(\ep) +o(W(\ep)) \\
&\quad+\left(1-\frac{q}{\Phi(p+q+r)-\Phi(p+r)}e^{-\Phi(p+q+r)\ep}W^{(p+r)}(\ep) \right)\omega^+_{1,\ep}(0). \\
\end{split}
\end{equation}
%where we have used the fact that $\lim_{\ep\goto 0+}W^{(q)}(\ep)/W(\ep)=1 $; see \cite{KKR13}.

Solve (\ref{equB1}) for $\omega^+_\ep(0)$ and take a limit. Then
\begin{equation*}
\begin{split}
\omega^+_1(0)&=\lim_{\ep\goto 0+}\omega^+_{1,\ep}(0) \\
&=\lim_{\ep\goto 0+}\frac{e^{-\Phi(p+q+r)\ep}\frac{r}{\Phi(r)}W^{(r)}(\ep) }{\frac{q}{\Phi(p+q+r)-\Phi(p+r)}e^{-\Phi(p+q+r)\ep}W^{(p+r)}(\ep)}\\
&=\frac{r}{q\Phi(r)}(\Phi(p+q+r)-\Phi(p+r)).
\end{split}
\end{equation*}

Similarly, for $\theta> 0$,
\begin{equation*}
\begin{split}
&\bE[e^{-pT^+_\ep(r)+\theta X(T^+_\ep(r))}; T^+_\ep(r)=e_r, T_1<\tau^+_\ep\wedge e_r, A^+_\ep ]\\
&\leq\bP\{ T_1<\tau^+_\ep, 0\leq X_{T_1}\leq \ep \}\\
&=o(W(\ep)).
\end{split}
\end{equation*}
Then for $\theta>p+q+r$,
\begin{equation}\label{equB2}
\begin{split}
\omega^+_{2,\ep}(0)
&=\bE[e^{-p\tau^+_\ep};\tau^+_\ep<e_q\wedge e_r]\int_{-\infty}^0 \bE_\ep[e^{-p\tau^-_0}; \tau^-_0<e_r, X_{\tau^-_0}\in dy]\\
&\qquad\times\left(\bE_y[e^{-p\tau^+_0}; \tau^+_0<e_q\wedge e_r]\omega^+_{2,\ep}(0)+\bE_y[e^{-pe_r+\theta X(e_r)}; e_r<\tau^+_0\wedge e_q]\right)  \\
&\quad+\bE[e^{-pe_r+\theta X(e_r)}; e_r<\tau^+_\ep\wedge e_q, X_{e_r}<0]+o(W(\ep))\\
&=e^{-\Phi(p+q+r)\ep}\bE_\ep e^{-(p+r)\tau^-_0+\Phi(p+q+r) X(\tau^-_0)}\omega^+_{2,\ep}(0)\\
&\quad+e^{-\Phi(p+q+r)\ep}\frac{r}{p+q+r-\psi(\theta)}\left(\bE_\ep e^{-(p+r)\tau^-_0+\theta X(\tau^-_0)}-\bE_\ep e^{-(p+r)\tau^-_0+\Phi(p+q+r) X(\tau^-_0)} \right)  \\
&\quad+\frac{r}{p+q+r-\psi(\theta)}(1-e^{(\theta-\Phi(p+q+r))\ep}) +o(W(\ep))\\
&=\left(1-\frac{q}{\Phi(p+q+r)-\Phi(p+r)}e^{-\Phi(p+q+r)\ep}W^{(p+r)}(\ep) \right)\omega^+_{2,\ep}(0)+o(W(\ep)) \\
&\quad+\frac{r}{p+q+r-\psi(\theta)}\left(\frac{q}{\Phi(p+q+r)-\Phi(p+r)}W^{(p+r)}(\ep) \right.\\
&\qquad\qquad\qquad\qquad\left.-e^{(\theta-\Phi(p+q+r))\ep}\frac{p+r-\psi(\theta)}{\Phi(p+r)-\theta}W^{(p+r)}(\ep) \right),\\
\end{split}
\end{equation}
where by (\ref{idenB})
\begin{equation*}
\begin{split}
&\bE[e^{-pe_r+\theta X(e_r)}; e_r<\tau^+_\ep\wedge e_q, X_{e_r}<0]\\
&=\bE[e^{-(p+q)e_r+\theta X(e_r)}; e_r<\tau^+_\ep]\\
&=\int_{-\infty}^0 e^{\theta y}r\left(e^{-\Phi(p+q+r)\ep}W^{(p+q+r)}(\ep-y)-W^{(p+q+r)}(-y) \right)dy\\
&=re^{-\Phi(p+q+r)\ep}\int_{-\infty}^{-\ep}e^{\theta(y+\ep)}W^{(p+q+r)}(-y)dy-r\int_{-\infty}^0 e^{\theta y}W^{(p+q+r)}(-y)dy\\
&=e^{(\theta-\Phi(p+q+r))\ep}\frac{r}{\psi(\theta)-p-q-r}-\frac{r}{\psi(\theta)-p-q-r}+o(W(\ep)).
\end{split}
\end{equation*}
%and
%\begin{equation*}
%\begin{split}
%&\bE_y[e^{-pe_r+\theta X(e_r)}; e_r<\tau^+_0\wedge e_q]\\
%&=\bE_y[e^{-(p+q)e_r+\theta X(e_r)}; e_r<\tau^+_0]\\
%&=\int_{-\infty}^0 e^{\theta z}r\left( e^{\Phi(p+q+r)y}W^{(p+q+r)}(-z)-W^{(p+q+r)}(y-z)\right)\\
%&=e^{\Phi(p+q+r)y}\frac{r}{\psi(\theta)-p-q-r}-e^{\theta y}\frac{r}{\psi(\theta)-p-q-r}\\
%\end{split}
%\end{equation*}

Then
\begin{equation}
\begin{split}
\omega^+_2(0)&=\lim_{\ep\goto 0+}\omega^+_{2,\ep}(0)\\
&=\frac{\frac{r}{p+q+r-\psi(\theta)}\left(\frac{q}{\Phi(p+q+r)-\Phi(p+r)}-e^{(\theta-\Phi(p+q+r))\ep}\frac{p+r-\psi(\theta)}{\Phi(p+r)-\theta}\right) }{\frac{q}{\Phi(p+q+r)-\Phi(p+r)}}\\
&=\frac{r}{p+q+r-\psi(\theta)}\left(1-\frac{p+r-\psi(\theta)}{q}\times\frac{\Phi(p+q+r)-\Phi(p+r)}{\Phi(p+r)-\theta}\right).\\
%&=\frac{r}{p+q+r-\psi(\theta)}\left(\frac{p+q+r}{q}-\frac{(p+r)\Phi(p+q+r)}{q\Phi(p+r)}\right).
\end{split}
\end{equation}

For $x<0$,
\begin{equation*}\label{omega1a}
\begin{split}
\omega^+_1(x)&=\bE_x[e^{-(p+q)\tau^+_0}; \tau^+_0<e_r]\omega^+_1(0)\\
&=e^{\Phi(p+q+r)x}\frac{r}{q\Phi(r)}(\Phi(p+q+r)-\Phi(p+r)).
\end{split}
\end{equation*}
Similarly,
\begin{equation}\label{omega2a}
\begin{split}
\omega^+_2(x)&=\bE_x[e^{-(p+q)e_r+\theta X(e_r)}; e_r<\tau^+_0]+  \bE_x[e^{-(p+q)\tau^+_0}; \tau^+_0<e_r]\omega^+_2(0)  \\
%&=
%e^{\Phi(p+q+r)x}\frac{r}{\psi(\theta)-p-q-r}-e^{\theta x}\frac{r}{\psi(\theta)-p-q-r}\\
%&\quad+\frac{re^{\Phi(p+q+r)x}}{p+q+r-\psi(\theta)}\left(1-\frac{p+r-\psi(\theta)}{q}\times\frac{\Phi(p+q+r)-\Phi(p+r)}{\Phi(p+r)-\theta}\right)\\
&=\frac{r}{p+q+r-\psi(\theta)}\left(e^{\theta x}-e^{\Phi(p+q+r)x}\frac{p+r-\psi(\theta)}{q}\times\frac{\Phi(p+q+r)-\Phi(p+r)}{\Phi(p+r)-\theta}\right).
\end{split}
\end{equation}
By analytical extension, (\ref{omega2a}) holds for all $\theta\geq 0$.

%Identity (\ref{thm1a}) follows from a combination of (\ref{omega1a}) and (\ref{omega2a}).

For $x>0$,
\begin{equation*}
\begin{split}
\omega^+_1(x)&= \bP_x\{\tau^-_0>e_r\}+\int_{-\infty}^0\bE_x[e^{-p\tau^-_0}; \tau^-_0<e_r, X_{\tau^-_0}\in dy]\bE_y[e^{-(p+q)\tau^+_0}; \tau^+_0<e_r]\omega^+_1(0)\\
&=\frac{r}{\Phi(r)}W^{(r)}(x)-r\int_0^x W^{(r)}(y)dy\\
&\quad+\frac{r(\Phi(p+q+r)-\Phi(p+r))}{q\Phi(r)}\left\{Z^{(p+r)}(x, \Phi(p+q+r)) -\frac{q}{\Phi(p+q+r)-\Phi(p+r)}W^{(p+r)}(x) \right\}\\
&=\frac{r(W^{(r)}(x)-W^{(p+r)}(x) )}{\Phi(r)}-r\int_0^x W^{(r)}(y)dy\\
&\quad+\frac{r(\Phi(p+q+r)-\Phi(p+r))}{q\Phi(r)}Z^{(p+r)}(x, \Phi(p+q+r))\\
\end{split}
\end{equation*}
and
\begin{equation*}
\begin{split}
\omega^+_2(x)
&=\int_{-\infty}^0\bE_x[e^{-p\tau^-_0}; \tau^-_0<e_r, X_{\tau^-_0}\in dy]\\
&\qquad\times\left(\bE_y[e^{-(p+q)\tau^+_0}; \tau^+_0<e_r]
\omega^+_2(0)+\bE_y[e^{-(p+q)e_r+\theta X(e_r)}; e_r<\tau^+_0 ] \right)\\
&=\bE_x e^{-(p+r)\tau^-_0+\Phi(p+q+r)X(\tau^-_0)}\omega^+_2(0)\\
&\quad+\frac{r}{p+q+r-\psi(\theta)}\left(\bE_x e^{-(p+r)\tau^-_0+\theta X(\tau^-_0)}-\bE_x e^{-(p+r)\tau^-_0+\Phi(p+q+r)X(\tau^-_0)}  \right)\\
%&=\left\{Z^{(p+r)}(x, \Phi(p+q+r))-\frac{q}{\Phi(p+q+r)-\Phi(p+r)}W^{(p+r)}(x)\right\}\\
%&\qquad\times \frac{r}{p+q+r-\psi(\theta)}\left(-\frac{p+r-\psi(\theta)}{q}\times\frac{\Phi(p+q+r)-\Phi(p+r)}{\Phi(p+r)-\theta}\right)\\
%&\quad+\frac{r}{p+q+r-\psi(\theta)}\left(Z^{(p+r)}(x,\theta)-\frac{p+r-\psi(\theta)}{\Phi(p+r)-\theta}W^{(p+r)}(x)   \right)\\
&=\frac{r}{p+q+r-\psi(\theta)}\left(-\frac{p+r-\psi(\theta)}{q}\times\frac{\Phi(p+q+r)-\Phi(p+r)}{\Phi(p+r)-\theta}\right) Z^{(p+r)}(x, \Phi(p+q+r))\\
&\quad +\frac{r}{p+q+r-\psi(\theta)}Z^{(p+r)}(x,\theta).\\
%&+\frac{r(p+r-\psi(\theta))}{p+q+r-\psi(\theta)}\left( \frac{1}{\Phi(p+r)}-\frac{1}{\Phi(p+r)-\theta}\right)W^{(p+r)}(x).  \\
\end{split}
\end{equation*}
The desired results then follow.
\end{proof}

\begin{rem}
For $x<0$,
\begin{equation}
\begin{split}
\bE_x e^{-pT^+(r)}&=\lim_{q,\theta\goto 0+}(\omega^+_1(0)+\omega^+_2(0))\\
%&=e^{\Phi(p+r)x}\frac{r}{\Phi(r)}\Phi'(p+r)+\frac{r}{p+r}\left( 1-e^{\Phi(p+r)x}\frac{p+r}{\Phi(p+r)}\Phi'(p+r)\right) \\
&=\frac{r}{p+r}+re^{\Phi(p+r)x}\Phi'(p+r)\left( \frac{1}{\Phi(r)}-\frac{1}{\Phi(p+r)}\right).
\end{split}
\end{equation}

For $x>0$,
\begin{equation}
\begin{split}
\bE_x e^{-pT^+(r)}%&=\frac{r(W^{(r)}(x)-W^{(p+r)}(x) )}{\Phi(r)}+\frac{r\Phi'(p+r)}{\Phi(r)}Z^{(p+r)}(x, \Phi(p+r))-r\int_0^x W^{(r)}(y)dy\\
%&\quad-\frac{r\Phi'(p+r)}{\Phi(p+r)} Z^{(p+r)}(x, \Phi(p+r)) +\frac{r}{p+r}Z^{(p+r)}(x,0) \\
&=\frac{r(W^{(r)}(x)-W^{(p+r)}(x) )}{\Phi(r)}+re^{\Phi(p+r)x}\Phi'(p+r)\left(\frac{1}{\Phi(r)}-\frac{1}{\Phi(p+r)}\right)\\
&\quad+1-Z^{(r)}(x)  +\frac{r}{p+r}Z^{(p+r)}(x), \\
\end{split}
\end{equation}
which recovers Theorem 2 in \cite{Bau09}.
\end{rem}

\begin{rem}
For $\psi'(0+)>0$,
letting $r, q\goto 0+$   we have
\[\bE_x e^{-pT^+}= \psi'(0+)(W(x)-W^{(p)}(x))+e^{\Phi(p)x}\psi'(0+)\Phi'(p),  \]
%\[\bE e^{-pT^+}=\psi'(0+)\Phi'(p)=\frac{\psi'(0+)}{\psi'(\Phi(p))},\]
which recovers  a result in \cite{CY05}; letting $r, p\goto
0+ $ for $\psi'(0+)>0$, we have
\[\bE_x e^{-q \int_0^\infty 1_{X_s <0} ds}=\frac{\psi'(0+)}{q}\Phi(q)Z(x,\Phi(q)), \]
%\[\bE e^{-q \int_0^\infty 1_{X_s <0} ds}=\frac{\psi'(0+)}{q}\Phi(q), \]
which recovers  Corollary 3 (ii) of \cite{LRZ14}.
\end{rem}

We also want to find
$$\omega^-_1(x):=\bE_x [e^{-pT^-(r)+\theta X_{T^-(r)}-q \int_0^{T^-(r)}  1_{X_s \geq 0} ds}; T^-(r)<e_r]$$
and
\[ \omega^-_2(x):=\bE_x [e^{-pT^-(r)-q \int_0^{T^-(r)}  1_{X_s \geq 0} ds-\theta X(T^-(r))}; T^-(r)=e_r].\]

\begin{thm}\label{thm2}
For any $p, q, r, \theta>0$ and $x\in \bR$, we have
\begin{equation}\label{down_a}
\begin{split}
%&\bE_x [e^{-pT^-(r)+\theta X_{T^-(r)}-q \int_0^{T^-(r)}  1_{X_s \geq 0} ds}; T^-(r)<e_r] \\
\omega^-_1(x)&= \frac{1}{q}(\Phi(p+q+r)-\Phi(p+r))Z^{(p+q+r)}(x,\Phi(p+r))\\
&\qquad\times
\left(\frac{p+q+r-\psi(\theta+\Phi(r))}{\Phi(p+q+r)-\theta-\Phi(r)}-\frac{p+q+r-\psi(\theta)}{\Phi(p+q+r)-\theta}\right)\\
&\quad+Z^{(p+q+r)}(x,\theta)-Z^{(p+q+r)}(x,\theta+\Phi(r))
\end{split}
\end{equation}
and
\begin{equation}
\begin{split}
%&\bE_x [e^{-pT^-(r)-q \int_0^{T^-(r)}  1_{X_s \geq 0} ds-\theta X(e_r)}; T^-(r)=e_r]\\
\omega^-_2(x)&=\frac{r}{q}\times\frac{\Phi(p+q+r)-\Phi(p+r)}{\theta+\Phi(p+q+r)}Z^{(p+q+r)}(x,\Phi(p+r))
-r\int_0^x e^{-\theta y}W^{(p+q+r)}(x-y)dy.
\end{split}
\end{equation}
\end{thm}

\begin{proof}
If $X$ is of bounded variation,  for
\[A^-:=\cap_{i: T_i<T^-(r)}\{ X_{T_i}< 0 \},\]
we have
\[\omega^-_{1}(x)= \bE_x\left[ e^{-pT^-(r)+\theta X_{T^-(r)}-q \int_0^{ T^-(r)} 1_{X_s \geq 0} ds}; T^-(r)<e_r, A^-\right].\]
Then
\begin{equation}
\begin{split}
\omega^-_1(0)&=\int^0_{-\infty} \bE[e^{-p\tau^-_0}; \tau^-_0<e_q\wedge e_r, X_{\tau^-_0}\in dy]\\
&\qquad\times\left( \bE_y[e^{-p\tau^+_0};\tau^+_0<e_r ]\omega^-_1(0)+ e^{\theta y}\bP_y\{ e_r<\tau^+_0 \}\right)\\
&=\int^0_{-\infty} \bE[e^{-(p+q+r)\tau^-_0}; X_{\tau^-_0}\in dy]\left(\bE_y e^{\Phi(p+r)y}\omega^-_1(0)+e^{\theta y}-e^{(\theta+\Phi(r))y} \right)\\
%&=\bE e^{-(p+q+r)\tau^-_0+\Phi(p+r)X(\tau^-_0)}\omega^-_1(0)\\
%&\quad+\bE e^{-(p+q+r)\tau^-_0+\theta X(\tau^-_0)}-\bE e^{-(p+q+r)\tau^-_0+(\theta+\Phi(r))X(\tau^-_0)}\\
&=\left(1-\frac{qW^{(p+q+r)}(0)}{\Phi(p+q+r)-\Phi(p+r)}\right)\omega^-_1(0)-\frac{p+q+r-\psi(\theta)}{\Phi(p+q+r)-\theta}W^{(p+q+r)}(0)\\
&\qquad+\frac{p+q+r-\psi(\theta+\Phi(r))}{\Phi(p+q+r)-\theta-\Phi(r)}W^{(p+q+r)}(0).\\
\end{split}
\end{equation}
Then
\begin{equation*}
\begin{split}
\omega^-_1(0)
%&=\frac{\frac{p+q+r-\psi(\theta+\Phi(r))}{\Phi(p+q+r)-\theta-\Phi(r)}-\frac{p+q+r-\psi(\theta)}{\Phi(p+q+r)-\theta}}{\frac{q}{\Phi(p+q+r)-\Phi(p+r)}}\\
&=\frac{1}{q}(\Phi(p+q+r)-\Phi(p+r))\left(\frac{p+q+r-\psi(\theta+\Phi(r))}{\Phi(p+q+r)-\theta-\Phi(r)}-\frac{p+q+r-\psi(\theta)}{\Phi(p+q+r)-\theta}\right).\\
\end{split}
\end{equation*}

Similarly,
\begin{equation*}
\begin{split}
\omega^-_2(0)&=\bE[e^{-pe_r-\theta X(e_r)}; e_r<\tau^-_0\wedge e_q]\\
&\quad+\int^0_{-\infty} \bE[e^{-p\tau^-_0}; \tau^-_0<e_q\wedge e_r, X_{\tau^-_0}\in dy]\bE_y[e^{-p\tau^+_0};\tau^+_0<e_r ]\omega^-_1(0)\\
&=\bE[e^{-(p+q)e_r-\theta X(e_r)}; e_r<\tau^-_0]+\bE e^{-(p+q+r)\tau^-_0+\Phi(p+r)\tau^-_0} \\
&=\frac{r}{\theta+\Phi(p+q+r)} W^{(p+q+r)}(0)+\left(1-\frac{qW^{(p+q+r)}(0)}{\Phi(p+q+r)-\Phi(p+r)}\right)\omega^-_2(0).\\
\end{split}
\end{equation*}
where by identity (\ref{idenC})
\begin{equation*}
\begin{split}
&\bE[e^{-(p+q)e_r-\theta X(e_r)}; e_r<\tau^-_0]\\
&=r\int_0^\infty e^{-\theta y}\int_0^\infty\bE[e^{-(p+q+r)t}; t<\tau^-_0, X_t\in dy]  \\
&=r\int_0^\infty e^{-\theta y-\Phi(p+q+r) y}W^{(p+q+r)}(0) dy\\
&=\frac{r}{\theta+\Phi(p+q+r)}W^{(p+q+r)}(0).
\end{split}
\end{equation*}
Then
\[\omega^-_2(0)=\frac{r}{q}\times\frac{\Phi(p+q+r)-\Phi(p+r)}{\theta+\Phi(p+q+r)}.\]

To deal with process $X$ with sample paths of unbounded variation,
 for $\ep>0$ and $x\in\bR$, let
\[T^-_\ep(r):=\inf\{T^-(r)\leq s\leq e_r: X_s< -\ep\}\]
with the convention $\inf\emptyset=e_r $ and
\[\omega^-_{1,\ep}(x):=\bE_x[ e^{-pT^-_\ep(r)+\theta X_{T^-_\ep(r)}-q \int_0^{ T^-_\ep(r)} 1_{X_s \geq 0} ds}; T^-_\ep(r)<e_q],\]
\[\omega^-_{2,\ep}(x):=\bE_x[ e^{-pT^-_\ep(r)-q \int_0^{ T^-_\ep(r)} 1_{X_s \geq 0} ds-\theta X(e_r)}; T^-_\ep(r)=e_q].\]
Then for
\[A^-_\ep:=\cap_{i:T_i<T^-_\ep(r)}\{ X_{T_i}< 0 \},\]
we have
\[\omega^-_{1,\ep}(x)= \bE_x\left[ e^{-pT^-_\ep(r)+\theta X_{T^-_\ep(r)}-q \int_0^{ T^-_\ep(r)} 1_{X_s \geq 0} ds}; T^-_\ep(r)<e_r, A^-_\ep\right].\]

By (\ref{idenC}),
\begin{equation*}
\begin{split}
&\bE [e^{-p T^-_\ep(r)+\theta X(T^-_\ep(r) )}; T_1<e_r\wedge \tau^-_{-\ep}, T^-_\ep(r)<e_r, A^-_\ep]\\
&\leq \bP\{ T_1<\tau^-_{-\ep}, -\ep\leq X_{T_1}\leq 0\}\\
&=o(W(\ep)).
\end{split}
\end{equation*}
Then
\begin{equation*}
\begin{split}
\omega^-_{1,\ep}(0)
&=\int_{-\infty}^{-\ep} \bE [e^{-p\tau^-_{-\ep}+\theta y}; \tau^-_{-\ep}<e_q\wedge e_r, X_{\tau^-_{-\ep}}\in dy ]\bP_y\{\tau^+_0>e_r\}  \\
&\quad+ \int_{-\infty}^{-\ep} \bE [e^{-p\tau^-_{-\ep}}; \tau^-_{-\ep}<e_q\wedge e_r, X_{\tau^-_{-\ep}}\in dy ]\bE_y [e^{-p\tau^+_0}; \tau^+_0<e_r]\omega^-_{1,\ep}(0) \\
&\quad+o(W(\ep)) \\
&=\int_{-\infty}^{-\ep} \bE [e^{-(p+q+r)\tau^-_{-\ep}+\theta y};  X_{\tau^-_{-\ep}}\in dy ](1-e^{\Phi(r)y}) \\
&\quad+ \int_{-\infty}^{-\ep} \bE [e^{-(p+q+r)\tau^-_{-\ep}};  X_{\tau^-_{-\ep}}\in dy ]e^{\Phi(p+r)y}\omega^-_{1,\ep}(0)+o(W(\ep)) \\
%&=\bE [e^{-(p+q+r)\tau^-_{-\ep}+\theta X_{\tau^-_{-\ep}}}]- \bE [e^{-(p+q+r)\tau^-_{-\ep}+(\theta+\Phi(r)) X_{\tau^-_{-\ep}}}]\\
%&\quad+  \bE [e^{-(p+q+r)\tau^-_{-\ep}+\Phi(p+r)  X_{\tau^-_{-\ep}}}]\omega^-_\ep(0)+o(W(\ep)) \\
&=1+o(W(\ep))-\frac{p+q+r-\psi(\theta)}{\Phi(p+q+r)-\theta}e^{-\theta\ep}W^{(p+q+r)}(\ep)\\
&\quad-1+o(W(\ep))+\frac{p+q+r-\psi(\theta+\Phi(r))}{\Phi(p+q+r)-\theta-\Phi(r)}e^{-(\theta+\Phi(r))\ep}W^{(p+q+r)}(\ep)\\
&\quad+\left(1-\frac{q}{\Phi(p+q+r)-\Phi(p+r)}e^{-\Phi(p+r)\ep}W^{(p+q+r)}(\ep)\right)\omega^-_{1,\ep}(0). \\
\end{split}
\end{equation*}
Then
\begin{equation*}
\begin{split}
\omega^-_1(0)
&=\lim_{\ep\goto 0+}\omega^-_{1,\ep}(0)\\
%&=\frac{\frac{p+q+r-\psi(\theta+\Phi(r))}{\Phi(p+q+r)-\theta-\Phi(r)}-\frac{p+q+r-\psi(\theta)}{\Phi(p+q+r)-\theta}}{\frac{q}{\Phi(p+q+r)-\Phi(p+r)}}\\
&=\frac{1}{q}(\Phi(p+q+r)-\Phi(p+r))\left(\frac{p+q+r-\psi(\theta+\Phi(r))}{\Phi(p+q+r)-\theta-\Phi(r)}-\frac{p+q+r-\psi(\theta)}{\Phi(p+q+r)-\theta}\right).
\end{split}
\end{equation*}

Similarly,
%\begin{equation*}
%\begin{split}
%&\bE [e^{-pT^-_{-\ep}(r)-\theta X(T^-_{-\ep}(r))}; T_1<\tau^-_{-\ep}\wedge e_r, T^-_\ep(r)=e_r, A^-_\ep]\\
%&\leq\bP\{T_1<\tau^-_{-\ep}\wedge e_r, -\ep\leq X_{T_1}\leq 0\}\\
%& =o(W(\ep)).
%\end{split}
%\end{equation*}
%Then
\begin{equation*}
\begin{split}
\omega^-_{2,\ep}(0)
&=\int_{-\infty}^{-\ep} \bE [e^{-p\tau^-_{-\ep}}; \tau^-_{-\ep}<e_q\wedge e_r, X_{\tau^-_{-\ep}}\in dy ]\bE_y [e^{-p\tau^+_0}; \tau^+_0<e_r]\omega^-_{2,\ep}(0)\\
&\quad+\bE[e^{-pe_r-\theta X(e_r)}; e_r<\tau^-_{-\ep}\wedge e_q] +o(W(\ep)) \\
%&=\int_{-\infty}^{-\ep} \bE [e^{-(p+q+r)\tau^-_{-\ep}};  X_{\tau^-_{-\ep}}\in dy ]e^{\Phi(p+r)y}\omega^-_{2,\ep}(0) \\
%&\quad+\frac{r}{\theta+\Phi(p+q+r)}W^{(p+q+r)}(\ep)+o(W(\ep))\\
&=\left(1-\frac{q}{\Phi(p+q+r)-\Phi(p+r)}e^{-\Phi(p+r)\ep}W^{(p+q+r)}(\ep) \right)\omega^-_{2,\ep}(0) \\
&\quad+\frac{r}{\theta+\Phi(p+q+r)}W^{(p+q+r)}(\ep)+o(W(\ep)),
%&\quad+\frac{r(p+q+r)}{(p+q+r-\psi(\theta))\Phi(p+q+r)} W^{(p+q+r)}(\ep)+o(W(\ep)) \\
\end{split}
\end{equation*}
where
\begin{equation*}
\begin{split}
&\bE[e^{-pe_r-\theta X(e_r)}; e_r<\tau^-_{-\ep}\wedge e_q]\\
%&=r\int_{-\ep}^\infty e^{-\theta y}\int_0^\infty e^{-(p+q+r)t}\bP\{t<\tau^-_{-\ep}, X_t\in dy\}dt\\
&=r\int_{-\ep}^\infty e^{-\theta y}\left(e^{-\Phi(p+q+r)(y+\ep)}W^{(p+q+r)}(\ep)- W^{(p+q+r)}(-y)\right)dy\\
&=\frac{r}{\theta+\Phi(p+q+r)}W^{(p+q+r)}(\ep)+o(W(\ep)).
\end{split}
\end{equation*}
%\begin{equation*}
%\begin{split}
%&\bE[e^{-pe_r+\theta X(e_r)}; e_r<\tau^-_{-\ep}\wedge e_q]\\
%&=\bE e^{-(p+q)e_r+\theta X(e_r)}-\bE[e^{-(p+q)e_r+\theta X(e_r)}; \tau^-_{-\ep}<e_r] \\
%&=\bE e^{-(p+q)e_r+\theta X(e_r)}-\int_{-\infty}^{-\ep} \bE[e^{-(p+q)\tau^-_{-\ep}}; \tau^-_{-\ep}<e_r, X_{\tau^-_{-\ep}}\in dy]
%\bE_y e^{-(p+q)e_r+\theta X(e_r)}\\
%&=\bE e^{-(p+q)e_r+\theta X(e_r)}-\bE e^{-(p+q+r)\tau^-_{-\ep}}\bE e^{-(p+q)e_r+\theta X(e_r)}\\
%&=\frac{r}{p+q+r-\psi(\theta)}\left(1-Z^{(p+q+r)}(\ep)+\frac{p+q+r}{\Phi(p+q+r)}W^{(p+q+r)}(\ep) \right)\\
%&= \frac{r(p+q+r)}{(p+q+r-\psi(\theta))\Phi(p+q+r)}  W^{(p+q+r)}(\ep)+o(W(\ep)).
%\end{split}
%\end{equation*}
Then
\[\omega^-_2(0)=\frac{r}{q}\times\frac{\Phi(p+q+r)-\Phi(p+r)}{\theta+\Phi(p+q+r)}. \]
%\[\omega^-_2(0)=\frac{r(p+q+r)}{q(p+q+r-\psi(\theta))}\times\frac{\Phi(p+q+r)-\Phi(p+r)}{\Phi(p+q+r)}. \]

For $x<0$,
\begin{equation*}
\begin{split}
\omega^-_1(x)&=e^{\theta x}\bP_x\{e_r<\tau^+_0\}+\bE_x [e^{-p\tau^+_0}; \tau^+_0<e_r]\omega^-_1(0)  \\
&=e^{\theta x}(1-e^{\Phi(r)x})\\
&\quad+\frac{e^{\Phi(p+r)x}}{q}(\Phi(p+q+r)-\Phi(p+r))\left(\frac{p+q+r-\psi(\theta+\Phi(r))}{\Phi(p+q+r)-\theta-\Phi(r)}-\frac{p+q+r-\psi(\theta)}{\Phi(p+q+r)-\theta}\right). \\
\end{split}
\end{equation*}

\begin{equation*}
\begin{split}
\omega^-_2(x)&= \bE_x [e^{-p\tau^+_0}; \tau^+_0<e_r]\omega^-_2(0)  \\
&=\frac{re^{\Phi(p+r)x}}{q}\times\frac{\Phi(p+q+r)-\Phi(p+r)}{\theta+\Phi(p+q+r)}.
\end{split}
\end{equation*}

For $x>0$,
\begin{equation*}
\begin{split}
\omega^-_1(x)&=\int_{-\infty}^0\bE_x[ e^{-(p+q)\tau^-_0}; \tau^-_0<e_r, X_{\tau^-_0}\in dy ]\left(\bE_y[e^{-p\tau^+_0};\tau^+_0<e_r]
\omega^-_1(0)+e^{\theta y}\bP_y\{e_r<\tau^+_0\} \right)\\
&=\bE_x e^{-(p+q+r)\tau^-_0+\Phi(p+r) X(\tau^-_0)}\omega^-_1(0)+\bE_x e^{-(p+q+r)\tau^-_0+\theta X(\tau^-_0)}
-\bE_x e^{-(p+q+r)\tau^-_0+(\theta+\Phi(r)) X(\tau^-_0)}   \\
&=\left\{ Z^{(p+q+r)}(x,\Phi(p+r))-\frac{q}{\Phi(p+q+r)-\Phi(p+r)}W^{(p+q+r)}(x) \right\}\\
&\qquad\times\frac{1}{q}(\Phi(p+q+r)-\Phi(p+r))\left(\frac{p+q+r-\psi(\theta+\Phi(r))}{\Phi(p+q+r)-\theta-\Phi(r)}
-\frac{p+q+r-\psi(\theta)}{\Phi(p+q+r)-\theta}\right)\\
&\quad+Z^{(p+q+r)}(x,\theta)-\frac{p+q+r-\psi(\theta)}{\Phi(p+q+r)-\theta}W^{(p+q+r)}(x)\\
&\quad-Z^{(p+q+r)}(x,\theta+\Phi(r))+\frac{p+q+r-\psi(\theta+\Phi(r))}{\Phi(p+q+r)-\theta-\Phi(r)}W^{(p+q+r)}(x)\\
&= Z^{(p+q+r)}(x,\Phi(p+r))\omega^-_1(0)+Z^{(p+q+r)}(x,\theta)-Z^{(p+q+r)}(x,\theta+\Phi(r)).\\
\end{split}
\end{equation*}
Similarly,
\begin{equation*}
\begin{split}
\omega^-_2(x)&=\int_{-\infty}^0\bE_x[ e^{-(p+q)\tau^-_0}; \tau^-_0<e_r, X_{\tau^-_0}\in dy ]\bE_y[e^{-p\tau^+_0};\tau^+_0<e_r]\omega^-_2(0)
+\bE_x[e^{-(p+q)e_r-\theta X(e_r)}; e_r<\tau^-_0] \\
&=\frac{r}{q}\times\frac{\Phi(p+q+r)-\Phi(p+r)}{\theta+\Phi(p+q+r)}\left\{ Z^{(p+q+r)}(x,\Phi(p+r))-\frac{q}{\Phi(p+q+r)-\Phi(p+r)}W^{(p+q+r)}(x) \right\}\\
&\quad+\frac{r}{\theta+\Phi(p+q+r)}W^{(p+q+r)}(x)-r\int_0^x e^{-\theta y}W^{(p+q+r)}(x-y)dy\\
&=\frac{r}{q}\times\frac{\Phi(p+q+r)-\Phi(p+r)}{\theta+\Phi(p+q+r)}Z^{(p+q+r)}(x,\Phi(p+r))
-r\int_0^x e^{-\theta y}W^{(p+q+r)}(x-y)dy,
\end{split}
\end{equation*}
where
\begin{equation*}
\begin{split}
&\bE_x[e^{-(p+q)e_r-\theta X(e_r)}; e_r<\tau^-_0]\\
&=r\int_0^\infty e^{-\theta y}\left(e^{-\Phi(p+q+r)y}W^{(p+q+r)}(x)-W^{(p+q+r)}(x-y)  \right)dy  \\
&=\frac{r}{\theta+\Phi(p+q+r)}W^{(p+q+r)}(x)-r\int_0^x e^{-\theta y}W^{(p+q+r)}(x-y)dy.
\end{split}
\end{equation*}
We thus finish the proof.
\end{proof}

Using Theorem \ref{thm2} we can find the probability that $X$ hits $0$ at time $T^-(r)$. It shows that $\bP_x\{X_{T^-(r)}=0\}>0$ if and only if process $X$ has a nontrivial Brownian component, which is similar to what happens at the first exit time.

\begin{cor}\label{corA}
For any $p,q,t>0$ and $x\in\bR$, we have
\begin{equation}\label{creeping}
\begin{split}
&\bE_x [e^{-pT^-(r)-q \int_0^{T^-(r)}  1_{X_s \geq 0} ds}; T^-(r)<e_r, X_{T^-(r)}=0] \\
&= \frac{\sigma^2}{2} \Phi(r)Z^{(p+q+r)}(x,\Phi(p+r)) \frac{1}{q}(\Phi(p+q+r)-\Phi(p+r))-\frac{\sigma^2}{2} \Phi(r)W^{(p+q+r)}(x).\\
\end{split}
\end{equation}
In particular, for $\psi'(0+)<0$,
\[\bE_x[e^{-pT^-};  X_{T^-}=0]=\frac{\sigma^2\Phi(0)}{2}\left(e^{\Phi(p)x}\Phi'(p)-W^{(p)}(x)\right).\]
%\[\bP_x \{ X_{T^-}=0\}=\frac{\sigma^2\Phi(0)}{2}\left(e^{\Phi(0)x}\Phi'(0)-W(x)\right).\]
\end{cor}

\begin{proof}
Notice that
\begin{equation}
\begin{split}
&\lim_{\theta\goto\infty}\left(\frac{p+q+r-\psi(\theta+\Phi(r))}{\Phi(p+q+r)-\theta-\Phi(r)}-\frac{p+q+r-\psi(\theta)}{\Phi(p+q+r)-\theta}\right)\\
&=\lim_{\theta\goto\infty}\frac{-\psi(\theta+\Phi(r))(\Phi(p+q+r)-\theta)+\psi(\theta)(\Phi(p+q+r)-\theta)+\Phi(r)(p+q+r-\psi(\theta))}{(\Phi(p+q+r)-\theta-\Phi(r))(\Phi(p+q+r)-\theta)}\\
%&=\lim_{\theta\goto\infty}\frac{-\psi(\theta+\Phi(r))\Phi(p+q+r)+\theta(\psi(\theta+\Phi(r))-\psi(\theta))+\psi(\theta)(\Phi(p+q+r)-\Phi(r)) }{\theta^2}\\
&=-\frac{\sigma^2}{2}\Phi(p+q+r)+\sigma^2 \Phi(r)+\frac{\sigma^2}{2}(\Phi(p+q+r)-\Phi(r))\\
&=\frac{\sigma^2}{2} \Phi(r),
\end{split}
\end{equation}
see Section 2.4 of \cite{KKR13}.

In addition,
\begin{equation}
\begin{split}
&\lim_{\theta\goto\infty}\left(Z^{(p+q+r)}(x,\theta)-Z^{(p+q+r)}(x,\theta+\Phi(r))\right)\\
&=\lim_{\theta\goto\infty}\left(\bE_x e^{-(p+q+r)\tau^-_0+\theta X(\tau^-_0)}1_{\tau^-_0<\infty}+\frac{p+q+r-\psi(\theta)}{\Phi(p+q+r)-\theta}W^{(p+q+r)}(x)\right.\\
&\qquad\qquad\left.- \bE_x e^{-(p+q+r)\tau^-_0+(\theta+\Phi(r)) X(\tau^-_0)}1_{\tau^-_0<\infty}
-\frac{p+q+r-\psi(\theta+\Phi(r))}{\Phi(p+q+r)-\theta-\Phi(r)}W^{(p+q+r)}(x) \right)\\
&=-\frac{\sigma^2}{2} \Phi(r)W^{(p+q+r)}(x).
\end{split}
\end{equation}
Then (\ref{creeping}) follows from (\ref{down_a}).
\end{proof}

\begin{cor}
For $\psi'(0+)<0$ and $x\in\bR$,
\begin{equation}
\begin{split}
&\bE_x e^{-pT^-+\theta X_{T^-}-q \int_0^{T^-}  1_{X_s \geq 0} ds}\\
&=Z^{(p+q)}(x,\Phi(p)) \frac{1}{q}(\Phi(p+q)-\Phi(p))\left(\frac{p+q-\psi(\theta+\Phi(0))}{\Phi(p+q)-\theta-\Phi(0)}-\frac{p+q-\psi(\theta)}{\Phi(p+q)-\theta}\right)\\
&\quad+Z^{(p+q)}(x,\theta)-Z^{(p+q)}(x,\theta+\Phi(0)).
\end{split}
\end{equation}
\end{cor}

%\begin{rem}
%Letting $q, \theta\goto 0+$ we have
%\begin{equation*}
%\begin{split}
%\bE e^{-pT^-(r)}&=\bE [e^{-pT^-(r)}; T^-(r)<e_r]+\bE [e^{-pT^-(r)}; T^-(r)=e_r] \\
%&=\frac{\Phi'(p+r)}{\Phi(p+r)}\times\frac{(p+r)\Phi(r)-r\Phi(p+r)}{\Phi(p+r)-\Phi(r)}+\frac{r\Phi'(p+r)}{\Phi(p+r)} \\
%&=\frac{\Phi'(p+r)}{\Phi(p+r)}\times\frac{p\Phi(r)}{\Phi(p+r)-\Phi(r)},
%\end{split}
%\end{equation*}
%which agrees with Theorem 2 of \cite{Bau09}.
%\end{rem}

\begin{rem}
For $x<0$,
\begin{equation*}
\begin{split}
\bE_x e^{-pT^-(r)}&=1-e^{\Phi(r)x}+e^{\Phi(p+r)x}\frac{\Phi'(p+r)}{\Phi(p+r)}\times\frac{p\Phi(r)}{\Phi(p+r)-\Phi(r)}.
\end{split}
\end{equation*}

For $x>0$,
\begin{equation*}
\begin{split}
\bE_x e^{-pT^-(r)}&=Z^{(p+r)}(x,\Phi(p+r))\frac{\Phi'(p+r)}{\Phi(p+r)}\times\frac{(p+r)\Phi(r)-r\Phi(p+r)}{\Phi(p+r)-\Phi(r)}+Z^{(p+r)}(x)-Z^{(p+r)}(x,\Phi(r))\\
&\quad+\frac{r\Phi'(p+r)}{\Phi(p+r)}Z^{(p+r)}(x,\Phi(p+r))
-r\int_0^x W^{(p+r)}(x-y)dy\\
&=e^{\Phi(p+r)x}\frac{\Phi'(p+r)}{\Phi(p+r)}\times\frac{p\Phi(r)}{\Phi(p+r)-\Phi(r)}+\frac{r}{p+r}+\frac{p}{p+r}Z^{(p+r)}(x)-Z^{(p+r)}(x,\Phi(r)).\\
\end{split}
\end{equation*}
We thus recover Theorem 2 of \cite{Bau09}.
\end{rem}

\begin{rem}\label{remA}
For $\psi'(0+)<0$ and $x\in\bR$, letting $q, \theta, r\goto 0+$ we have
\[\bE_x e^{-pT^-}=Z^{(p)}(x,\Phi(p))\frac{p\Phi(0)\Phi'(p)}{\Phi(p)(\Phi(p)-\Phi(0))}+Z^{(p)}(x)-Z^{(p)}(x,\Phi(0));\]
letting $p, q, r\goto 0+$, we have
\[\bE_x e^{\theta X_{T^-}}=e^{\Phi(0)x}\Phi'(0)\left(\frac{\psi(\theta+\Phi(0))}{\theta}+\frac{\psi(\theta)}{\Phi(0)-\theta}\right)
+Z(x,\theta)-Z(x,\theta+\Phi(0));  \]
%\[\bE e^{\theta X_{T^-}}=\Phi'(0)\left(\frac{\psi(\theta+\Phi(0))}{\theta}+\frac{\psi(\theta)}{\Phi(0)-\theta}\right).  \]
letting $p, \theta, r\goto 0+ $, we have
\begin{equation*}
\begin{split}
\bE_x e^{-q \int_0^\infty 1_{X_s >0} ds}&=Z^{(q)}(x,\Phi(0))\frac{\Phi(0)}{\Phi(q)}+Z^{(q)}(x,0)-Z^{(q)}(x,\Phi(0)) \\
&=Z^{(q)}(x,\Phi(0))\left(\frac{\Phi(0)}{\Phi(q)}-1\right)+Z^{(q)}(x).
\end{split}
\end{equation*}
%\[\bE e^{-q \int_0^\infty 1_{X_s >0} ds}=\frac{\Phi(0)}{\Phi(q)},  \]
which recovers Corollary 3 (iv) of \cite{LRZ14}.
\end{rem}

We now consider some examples for more explicit results.
Let $X_t=-\mu t+B_t$ with $\mu>0$ and Brownain motion $B$.  Then $T^-<\infty$ a.s.,
\[\psi(\la)=-\mu\la+\frac{\la^2}{2}, \,\, \Phi(q)=\sqrt{\mu^2+2q}+\mu,\]
\[W^{(q)}(x)=\frac{1}{\sqrt{\mu^2+2q}}\left( e^{(\sqrt{\mu^2+2q}+\mu)x}-e^{-(\sqrt{\mu^2+2q}-\mu)x}\right)1_{x\geq 0}\]
 and
 \[W(x)=\frac{1}{\mu}(e^{2\mu x}-1)1_{x\geq 0}.\]
%\begin{equation}
%\begin{split}
%\bP_x \{ X_{T^-}=0\}&=\frac{\Phi(0)}{2}\left(e^{\Phi(0)x}\Phi'(0)-W(x)\right)\\
%&=\mu \left( e^{2\mu x} \frac{1}{\mu}-\frac{1}{\mu}(e^{2\mu x}-1)1_{x\geq 0}\right)\\
%&=e^{2\mu x} 1_{x<0}+1_{x\geq 0}.
%\end{split}
%\end{equation}
By Corollary \ref{corA},
\begin{equation}
\begin{split}
&\bE_x[e^{-pT^-};  X_{T^-}=0]\\
&=\frac{\Phi(0)}{2}\left(e^{\Phi(p)x}\Phi'(p)-W^{(p)}(x)\right)\\
&=\mu\left(\frac{1}{\sqrt{\mu^2+2p}} e^{(\sqrt{\mu^2+2p}+\mu)x}-\frac{1}{\sqrt{\mu^2+2p}}\left( e^{(\sqrt{\mu^2+2p}+\mu)x}-e^{-(\sqrt{\mu^2+2p}-\mu)x}\right)1_{x\geq 0}  \right)\\
&=\frac{\mu}{\sqrt{\mu^2+2p}}\left( e^{(\sqrt{\mu^2+2p}+\mu)x}1_{x<0}+e^{-(\sqrt{\mu^2+2p}-\mu)x}1_{x\geq 0} \right).\\
\end{split}
\end{equation}
It follows that
\[\bP_x\{  X_{T^-}=0\}= e^{2\mu x}1_{x<0}+1_{x\geq 0}=\bP_x\{\tau^+_0<\infty\} 1_{x<0}+1_{x\geq 0}, \]
which is expected intuitively.

Let $X$ be a spectrally negative $\alpha$-stable process with a negative drift whose Laplace exponent is given by $\psi(\lambda)=-\mu\la+\la^\al$ for $\mu>0$ and $1<\al<2$. Then $\Phi(0)=\mu^{\frac{1}{\al-1}}$ and by Remark \ref{remA}.
\begin{equation}
\begin{split}
\bE e^{\theta X_{T^-}}
&=\frac{1}{\psi'(\Phi(0))}\left(\frac{\psi(\theta+\Phi(0))}{\theta}+\frac{\psi(\theta)}{\Phi(0)-\theta}\right)\\
&=\frac{1}{(\al-1)\mu}\left(\frac{-\mu\theta-\mu^{\frac{\al}{\al-1}}+(\theta+\mu^{\frac{1}{\al-1}})^\al}{\theta}
+\frac{-\mu\theta+\theta^\al}{\mu^{\frac{1}{\al-1}}-\theta}  \right).\\
\end{split}
\end{equation}
It is not hard to see that
\[\lim_{\theta\goto\infty}\bE e^{\theta X_{T^-}}=0, \]
i.e. $\bP\{X_{T^-}=0\}=0$ which agrees with Corollary \ref{corA}.

 Let $X$ be a compound Poisson process with a positive drift $\mu$ and  negative exponential jumps with common rate $\rho$ arriving according to a Poisson process with rate $a$. We further assume that  $\mu-\frac{a}{\rho}<0$ so that the last exit time is finite. Then
\[\psi(\lambda)=\mu\lambda-\frac{a\la}{\rho+\la}, \,\,\,\,\Phi(0)=\frac{a}{\mu}-\rho,\]
\[ \psi'(\Phi(0))=\mu-\frac{a\rho}{\left(\rho+\frac{a}{\mu}-\rho\right)^2}=\mu-\frac{\rho\mu^2}{a}\]
and
\begin{equation}
\begin{split}
\psi(\theta+\Phi(0))&=\mu(\theta+\frac{a}{\mu}-\rho)-\frac{a(\theta+\frac{a}{\mu}-\rho)}{\rho+\theta+\frac{a}{\mu}-\rho} \\
&=(\theta+\frac{a}{\mu}-\rho)\frac{\mu\theta}{\theta+\frac{a}{\mu}}.
\end{split}
\end{equation}
It follows that
\begin{equation}
\begin{split}
\bE e^{\theta X_{T^-}}
&=\frac{1}{\psi'(\Phi(0))}\left(\frac{\psi(\theta+\Phi(0))}{\theta}+\frac{\psi(\theta)}{\Phi(0)-\theta}\right)\\
&=\frac{1}{\mu-\frac{\rho\mu^2}{a}}\left((\theta+\frac{a}{\mu}-\rho)\frac{\mu}{\theta+\frac{a}{\mu}}+\frac{\mu\theta-
\frac{a\theta}{\rho+\theta}}{\frac{a}{\mu}-\rho-\theta} \right)\\
%&=\frac{1}{\mu-\frac{\rho\mu^2}{a}}\left(\mu-\frac{\mu\rho}{\theta+\frac{a}{\mu}}-\frac{\mu\theta}{\rho+\theta}  \right)\\
%&=\frac{a}{a-\rho\mu}\left(\frac{\rho}{\rho+\theta}-\frac{\rho}{\theta+\frac{a}{\mu}}  \right)\\
&=\frac{a\rho}{(\rho+\theta)(\mu\theta+a)}\\
&=\frac{a}{a-\mu\rho}\times\frac{\rho}{\rho+\theta}-\frac{\mu\rho}{a-\mu\rho}\times\frac{a}{a+\mu\theta}.\\
\end{split}
\end{equation}
Therefore,  $X_{T^-}$ has a density
\[\left(\frac{a}{a-\mu\rho}\rho e^{\rho x}-\frac{\mu\rho}{a-\mu\rho}\times\frac{a}{\mu}e^{\frac{ax}{\mu}}\right)1_{x<0}=
\frac{a\rho}{a-\mu\rho}\left(e^{\rho x}-e^{\frac{ax}{\mu}}\right)1_{x<0}.\]
 It is remarkable that $X_{T^-}$ does not follow an exponential distribution. In addition, compared with the exponential random variable $e_\rho$, $-X_{T^-}$ has a better chance to take large values. Notice that  $X_{T^-}$ does not have the same density of $e_\rho$, the Laplace transform of $X_{T^-}$ can not be directly obtained from the distribution of $X_{T^--}$ due to the the fact that $T^-$ is not a stopping time.

%We also want to find
%$$\omega^+_3(x):=\bE_x [e^{-pT^+-q \int_0^{T^+}  1_{X_s \leq 0} ds}; T^+<\tau^-_a]$$
%and
%\[ \omega^-_3(x):=\bE_x [e^{-pT^--q \int_0^{T^-}  1_{X_s \geq 0} ds+\theta X(T^-)}; T^-<\tau^+_b].\]

%and
%\[??\bE_x [e^{-p\tau^-_a}Z^{(q)}(X_{\tau^-_a});\tau^-_a<\infty ]=Z_a^{(q,p)}(x)-W^{(p)}(x-a)\int_0^a e^{-\Phi(p)(y-a)} Z^{(q)}(y)dy. \]

\end{document}